\documentclass[11pt,letterpaper, reqno]{amsart}

\usepackage[american]{babel} 
\usepackage[utf8]{inputenc} 

\usepackage{xcolor}
\usepackage[shortlabels]{enumitem}
\usepackage{hyperref}
\hypersetup{colorlinks,linkcolor=.,citecolor=.,urlcolor=.}

\usepackage[T1]{fontenc}
\usepackage[sc]{mathpazo}
\usepackage{tgpagella}

\usepackage{amsmath}
\usepackage{amssymb}
\usepackage{amsthm}

\newtheorem{thm}{Theorem}[]
\newtheorem{lem}[thm]{Lemma}
\theoremstyle{definition}

\newcommand{\thistheoremname}{}
\newtheorem*{genericthm*}{\thistheoremname}
\newenvironment{namedtheorem}[1]
{\renewcommand{\thistheoremname}{#1}%
\begin{genericthm*}}
{\end{genericthm*}}

\newcommand{\QQ}{\mathbb{Q}}
\newcommand{\ZZ}{\mathbb{Z}}

\newcommand{\F}{\mathcal{F}}
\renewcommand{\S}{\mathcal{S}}

\newcommand{\norm}[1]{\left\lVert#1\right\rVert}
\newcommand{\abs}[1]{\lvert#1\rvert}

\newcommand{\qandq}{\quad\text{and}\quad}

\title{Existence of small ordered orthogonal arrays}

\author{Kai-Uwe Schmidt and Charlene Wei\ss}
\address{Department of Mathematics, Paderborn University, Warburger Str.\ 100, 33098 Paderborn, Germany.}
\email[K.-U. Schmidt]{kus@math.upb.de}
\email[C. Wei\ss]{chweiss@math.upb.de}

\date{03 September 2021 (revised 26 February 2023)}

\begin{document}

\begin{abstract}
We show that there exist ordered orthogonal arrays, whose sizes deviate from the Rao bound by a factor that is polynomial in the parameters of the ordered orthogonal array. The proof is nonconstructive and based on a probabilistic method due to Kuperberg, Lovett and Peled.\\[-7ex]
\end{abstract}
\maketitle


\section{Introduction}

A $t$-$(q,n,\lambda)$ \emph{orthogonal array} is an $M\times n$ array on $q$ symbols such that every $M\times t$ subarray contains each $t$-tuple on $q$ symbols exactly $\lambda$ times as a row. The parameter $t$ is called the \emph{strength} of the orthogonal array. These combinatorial objects were introduced in the 1940s and now have various applications, for example in statistics, coding theory, cryptography, and software testing. We refer to~\cite{SloaneOA} for background on orthogonal arrays and their applications. The complete set of $n$-tuples on~$q$ symbols is a $t$-$(q,n,\lambda)$ orthogonal array for every strength $t$. Therefore one is interested in the existence of orthogonal arrays with a fixed strength~$t$ having as few rows as possible. 
\par
Ordered orthogonal arrays generalise orthogonal arrays and were independently introduced by Lawrence~\cite{LawrenceOOA} and Mullen and Schmid~\cite{MullenSchmidOOA} in 1996. A \emph{$t$-$(q,n,r,\lambda)$ ordered orthogonal array} is an $M\times nr$ array on $q$ symbols, where the $nr$ columns are divided into $n$ blocks containing $r$ ordered columns such that for every $n$-tuple $(t_1,t_2,\dots,t_n)$ of integers summing up to $t$ with $0\le t_i\le r$, the rows of the $M\times t$ subarray consisting of the first $t_1$ columns of the first block, the first $t_2$ columns of the second block and so on, contain every $t$-tuple exactly~$\lambda$ times as a row. Note that $M=\lambda q^t$. We often say that $M$ is the \emph{size} of the array. An example for a $2$-$(2,2,2,1)$ ordered orthogonal array is
 \begin{align*}
 \begin{array}{cc|cc}
 0 & 0 & 0 & 0\\
 0 & 1 & 1 & 1\\
 1 & 0 & 1 & 0\\
 1 & 1 & 0 & 1
 \end{array}.
 \end{align*}
Observe that this is not an orthogonal array of strength $2$ since in the subarray consisting of the second and fourth column, the tuples $01$ and $10$ do not occur as rows.
\par
Again, ordered orthogonal arrays have numerous applications, in particular in coding theory and cryptography. Most notably, ordered orthogonal arrays are closely related to $(t,m,s)$-nets, which are of great significance in numerical integration, in the sense that a $(t,m,s)$-net in base $q$ exists if and only if an $(m-t)$-$(q,s,m-t,q^t)$ ordered orthogonal array exists~\cite{LawrenceOOA},~\cite{MullenSchmidOOA}.
\par
Similarly to orthogonal arrays, one is interested in having as few rows as possible. Apart from using $(t,m,s)$-nets (which produce $t$-$(q,n,t,\lambda)$ ordered orthogonal arrays), only a few constructions for ordered orthogonal arrays are known, see \cite{RosenbloomTsfasman}, \cite{Skriganov}, \cite{Castoldi}, \cite{Panario}, for example. These constructions produce MDS-like codes, namely optimal $t$-$(q,n,r,1)$ ordered orthogonal arrays of size $q^t$ in the case that $q$ is a prime power satisfying $q\ge n-1$. For $r=1$, they are MDS codes, hence optimal $t$-$(q,n,1)$ orthogonal arrays.
\par
Let $N(q,n,t)$ be the minimum number $M$ such that a $t$-$(q,n,\lambda)$ orthogonal array of size $M$ exists for some $\lambda$. Define $N(q,n,r,t)$ accordingly for ordered orthogonal arrays. Every $t$-$(q,n,r,\lambda)$ ordered orthogonal array gives a $t$-$(q,n,\lambda)$ orthogonal array by only choosing the first column in every block of the ordered orthogonal array. On the other hand, every $t$-$(q,nr,\lambda)$ orthogonal array gives a $t$-$(q,n,r,\lambda)$ ordered orthogonal array by dividing the $nr$ columns into $r$ blocks each of size $n$. Hence we have
\begin{equation}
\label{eqn:inequalities_N}
N(q,n,t)\le N(q,n,r,t)\le N(q,nr,t).
\end{equation}
Our main result is that, roughly speaking, the lower bound is more accurate than the upper bound if $n$ is large compared to $t$. A famous lower bound for $N(q,n,t)$ is given by the Rao bound~\cite{Rao}, which implies 
\[
\left(\frac{cqn}{t}\right)^{t/2}\le N(q,n,t)\quad\text{and}\quad \left(\frac{cqnr}{t}\right)^{t/2}\le N(q,nr,t),
\]
where $c>0$ is a universal constant independent of all other parameters. This shows in particular that
\begin{equation}
\label{eqn:N_lower_bound}
N(q,n,r,t)\ge \left(\frac{cqn}{t}\right)^{t/2}.
\end{equation}
We now state our main result.
\begin{thm}\label{thm:mainresult}
For all integers $q,n,r,t$ satisfying $q\geq 2$ and $1\leq t\leq nr$, there exists a $t$-$(q,n,r,\lambda)$ ordered orthogonal array~$Y$ such that
\begin{equation}
\label{eqn:N_upper_bound}
|Y|\leq \left(\frac{cq(n+t)}{t}\right)^{ct}
\end{equation}
for some universal constant $c>0$ independent of all other parameters.
\end{thm}
\par
We note that the lower bound~\eqref{eqn:N_lower_bound} has been strengthened by Martin and Stinson~\cite{MartinStinson_Rao}, \cite{MartinStinson_AS}. However this strengthened version involves quite complicated expressions and it is unclear as to whether this strengthened version is good enough to match the upper bound~\eqref{eqn:N_upper_bound} more accurately in the case that $t$ is large compared to $n$.
\par
We shall deduce Theorem~\ref{thm:mainresult} from a landmark result by Kuperberg, Lovett, and Peled~\cite{KLP}, which can be used to establish the existence of regular combinatorial structures. Their proof is based on probabilistic arguments and is therefore nonconstructive. The theorem was applied in~\cite{KLP} to show that nontrivial $t$-designs, orthogonal arrays of strength $t$, and $t$-wise permutations exist for all $t$. The so-called KLP theorem has been proved to be powerful in various other contexts. For example, Fazeli, Lovett, and Vardy \cite{FLV} used this result to prove the existence of nontrivial $q$-analogs of $t$-designs for all~$t$.
\par
In fact it follows from~\eqref{eqn:inequalities_N} and~\cite{KLP} that
\[
N(q,n,r,t)\le \left(\frac{cqnr}{t}\right)^{ct}
\]
for some universal constant $c>0$. We shall strengthen this result in order to prove Theorem~\ref{thm:mainresult}. To do so, we first recall the KLP theorem in the next section and then give a proof of Theorem~\ref{thm:mainresult} in Section~\ref{sec:OOA}.


\section{The KLP theorem}\label{sec:KLP}

In this section we recall the main theorem of~\cite{KLP}. Let $X$ be a finite set and let $V$ be a $\QQ$-linear subspace of functions $f\colon X\rightarrow\QQ$. We are interested in subsets $Y$ of $X$ satisfying
\begin{align}\label{eq:defDesign1}
\frac{1}{\abs{Y}}\sum_{x\in Y} f(x)=\frac{1}{\abs{X}}\sum_{x\in X}  f(x)\quad\text{for all }f\in V.
\end{align}
An \emph{integer basis} of $V$ is a basis of $V$ in which all elements are integer-valued functions. Let $\{\phi_a:a\in \F\}$ be an integer basis of $V$, where $\F$ is an index set. Then a subset~$Y$ of $X$ satisfies~\eqref{eq:defDesign1} if and only if
\begin{align}\label{eq:defDesign2}
\frac{1}{\abs{Y}}\sum_{x\in Y} \phi_a(x)=\frac{1}{\abs{X}}\sum_{x\in X}  \phi_a(x)\quad\text{for all }a\in \F.
\end{align}
The KLP theorem guarantees the existence of small subsets $Y$ of $X$ with this property, once the vector space satisfies five conditions. These conditions are recalled first.

\subsection*{Conditions}

\begin{enumerate}[label=(C\arabic*), font=\bfseries, leftmargin=3em]
\item \textbf{Constant Function.}
All constant functions belong to $V$, which means that every such function can be written as a rational linear combination of the basis functions $\phi_a$ with $a\in \F$.
\item \textbf{Symmetry.}
A permutation $\pi\colon X\to X$ is called a \emph{symmetry} of $V$ if $\phi_a\circ\pi$ lies in $V$ for all $a\in \F$. The set of symmetries of $V$ forms a group called the \emph{symmetry group} of $V$. The symmetry condition requires that the symmetry group acts transitively on $X$, which means that for all $x_1,x_2\in X$, there exists a symmetry $\pi$ such that $x_1=\pi(x_2)$.
\item \textbf{Divisibility.}
There exists a positive integer $c_1$ such that, for all $a\in \F$, there exists $n\in\ZZ^X$ (with $n=(n_x)_{x\in X}$) satisfying
\[
\frac{c_1}{\abs{X}}\sum_{x\in X} \phi_a(x)=\sum_{x\in X} n_x \phi_a(x)\quad\text{for every $a\in \F$}.
\]
The smallest positive integer $c_1$ for which this identity holds is called the \emph{divisibility constant} of~$V$.
\item \textbf{Boundedness of $V$.} 
The  $\ell_\infty$-norm of a function $g\colon X\rightarrow\QQ$ is given by 
\[
\norm{g}_\infty=\max_{x\in X} \lvert g(x)\rvert.
\]
The vector space $V$ has to be bounded in the sense that there exists a positive integer $c_2$ such that $V$ has a $c_2$-bounded integer basis in $\ell_\infty$.
\item \textbf{Boundedness of $V^\perp$.} 
The $\ell_1$-norm of a function $g\colon X\rightarrow \QQ$ is given by
\[
\norm{g}_1=\sum_{x\in X} \lvert g(x)\rvert.
\]
The orthogonal complement
\[
V^\perp=\left\{g\colon X\rightarrow\QQ:\sum_{x\in X} f(x)g(x)=0\;\text{for all }f\in V\right\}
\]
of $V$ has to be bounded in the sense that $V^\perp$ has a $c_3$-bounded integer basis in $\ell_1$. \end{enumerate}
\par
We can now state the KLP theorem.
\begin{namedtheorem}{KLP theorem}[{\cite[Theorem~2.4]{KLP}}]
Let $X$ be a finite set and let $V$ be a $\QQ$-linear subspace of functions $f\colon X\rightarrow \QQ$ satisfying the conditions (C1)--(C5) with the corresponding constants $c_1,c_2,c_3$. Let $N$ be an integral multiple of~$c_1$ with
\[
\min(N, \abs{X}-N)\geq C\,c_2c_3^2 (\dim V)^6 \log(2c_3\dim V)^6,
\]
where $C>0$ is a constant. Then there exists a subset $Y$ of $X$ of size $\abs{Y}=N$ such that
\[
\frac{1}{\abs{Y}}\sum_{x\in Y} f(x)=\frac{1}{\abs{X}}\sum_{x\in X}  f(x)\quad\text{for all }f\in V.
\]
\end{namedtheorem}
\par
We close this section with recalling a useful criterion for the verification of (C5) from~\cite{KLP}. An integer basis $\{\phi_a:a\in \F\}$ of~$V$ is \emph{locally decodable} if there exist functions $\gamma_a\colon X\rightarrow\ZZ$ such that
\begin{equation}\label{eqn:localdecodability}
\sum_{x\in X}\gamma_a(x)\phi_{a'}(x)=m\delta_{a,a'}\quad\text{for all $a,a'\in \F$}
\end{equation}
for some $m\in\ZZ$, where $\delta_{a,a'}$ denotes the Kronecker $\delta$-function. Note that $\{\gamma_a:a\in\F\}$ is necessarily an integer basis of $V$. If this basis is $c_4$-bounded in $\ell_1$, then we say that $\{\phi_a:a\in \F\}$ is locally decodable \emph{with bound $c_4$}.
\begin{lem}[{\cite[Claim~3.2]{KLP}}]
\label{lem:locally_decodable_condition}
Suppose that $\{\phi_a:a\in \F\}$ is a $c_2$-bounded integer basis in~$\ell_\infty$ of~$V$ that is locally decodable with bound $c_4$. Then $V^\perp$ has a $c_3$-bounded integer basis in $\ell_1$ with $c_3=2c_2c_4\abs{\F}$.
\end{lem}


\section{Proof of Theorem~\ref{thm:mainresult}}\label{sec:OOA}

In this section we prove Theorem~\ref{thm:mainresult} using the KLP theorem. Not surprisingly, our proof proceeds along similar lines as the proof given in \cite{KLP} for orthogonal arrays. We start by defining an ordered orthogonal array in the framework of the KLP theorem and specifying the underlying vector space~$V$. We then show that $V$ satisfies the conditions (C1)--(C5) with suitable constants, which establishes the existence of sufficiently small ordered orthogonal arrays.
\par
Henceforth we  denote by $[m]$ the set $\{1,2,\dots,m\}$. Next we define the set~$X$, the index set $\F$, and the vector space~$V$. Let $q,n,r,t$ be integers satisfying $n,r\geq 1$, $q\geq 2$, and $1\leq t\leq nr$.  Let $X$ be the set of all functions $[nr]\to [q]$.  We partition $[nr]$ into $n$ blocks of size $r$ containing subsequent numbers and let $\S$ be the family of $t$-subsets of $[nr]$ containing $t_i$ subsequent numbers from the $i$-th block, where $t_1,t_2,\dots,t_r\in\{0,1,\dots,r\}$ are integers summing up to $t$. Let $\F$ be the set of functions $S\to [q]$ with $S\in\S$ and, for $a\in\F$ with $a\colon S\to[q]$, define $\phi_a\colon X\to\QQ$ by
\[
\phi_a(x)=\begin{cases}
1 & \text{if $a(i)=x(i)$ for all $i\in S$,}\\
0 & \text{otherwise}.
\end{cases}
\]
Finally, let $V$ be the $\QQ$-span of~$\{\phi_a:a\in \F\}$. Now a subset~$Y$ of~$X$ is a $t$-$(q,n,r,\lambda)$ ordered orthogonal array if and only if~\eqref{eq:defDesign2} holds. Note that
\[
\frac{|Y|}{|X|}\sum_{x\in X} \phi_a(x)=\frac{|Y|}{q^t}=\lambda.
\]
In what follows we shall show that $V$ satisfies the conditions (C1)--(C5) with suitable constants and then deduce Theorem~\ref{thm:mainresult} from the KLP theorem.

\subsection*{(C1) Constant Function} 
For each $x\in X$, the sum
\begin{align}\label{eq:constfunc}
\sum_{a\in \F}\phi_a(x)=|\{a\in \F : a(i)=x(i)\text{ for all }i\in S\}|
\end{align}
is the cardinality of $\S$ since the image of $a$ is fixed by the image of~$x$. This implies 
\[
\frac{1}{|\S|}\sum_{a\in \F}\phi_a(x)=1
\]
for each $x\in X$ and hence $V$ contains the constant function.

\subsection*{(C2) Symmetry}
For each $x\in X$, define the permutation $\pi_x\colon X\to X$ by
\[
\pi_x(b)=x+b,
\]
where $x+b$ is the mapping in $X$ that satisfies $(x+b)(i)\equiv x(i)+b(i)\pmod q$ for all $i\in[nr]$. Then $\{\pi_x:x\in X\}$ is a group that acts transitively on $X$. We now show that this group is a subgroup of the symmetry group of $V$, which shows that the symmetry condition is satisfied. For each $b\in X$ and each $a\in \F$ with $a\colon S\to[q]$, we have
\[
\left(\phi_a\circ \pi_x\right)(b)=\phi_a(x+b)=\phi_{a'}(b),
\]
where $a'\in\F$ is the mapping $S\to[q]$ that satisfies $a'(i)\equiv a(i)-x(i)\pmod q$ for all $i\in S$. Therefore the function $\phi_a\circ \pi_x$ lies in~$V$, as required.

\subsection*{(C4) Boundedness of $V$} 
The set $\{\phi_a:a\in \F\}$ spans $V$ and consists of integer-valued functions that are $1$-bounded in $\ell_\infty$. Hence there exists a $c_2$-bounded integer basis of~$V$  with $c_2=1$.

\subsection*{(C5) Boundedness of $V^\perp$} 

We shall show that $V$ has a locally decodable integer basis with bound $2^t$. Lemma~\ref{lem:locally_decodable_condition} then implies that $V^\perp$ has a $c_3$-bounded integer basis in $\ell_1$ for $c_3=2^{t+1}\abs{\F}$.
\par
Recall that $a\in \F$ is a function $S\to[q]$ for some $t$-set $S\in\S$. Instead of taking the whole set $S$ as the domain of $a$, we now allow subsets of $S$. Moreover these subsets are now only mapped to $[q-1]$ instead of $[q]$. More formally, define
\[
\S'=\bigcup_{S\in\S}\bigcup_{T\subseteq S} T\qandq \F'=\{T\to [q-1]\colon T\in \S'\}.
\]
Note that, for each $b\in \F'$, there exists $a\in \F$ that coincides with $b$ if the domain of $a$ is restricted to that of $b$.
\par
First we will show that $V$ is spanned by $\{\phi_b:b\in \F'\}$.
\begin{lem}\label{lem:spanningsetV}
The set $\{\phi_b:b\in \F'\}$ spans $V$.
\end{lem}
\begin{proof}
We first show that every function~$\phi_b$ with $b\in \F'$ lies in $V$. To do so, let $b\in \F'$ with $b\colon T\to [q-1]$ for some $T\in \S'$ and choose some $S\in\S$ such that $T\subseteq S$. Consider the set $M$ of all mappings $a\colon S\to[q]$ that coincide with $b$ when their domains are restricted to $T$. Then, for every $x\in X$ with $\phi_b(x)=1$, there is exactly one element $a\in M$ with $\phi_a(x)=1$. Moreover, if $\phi_b(x)=0$, then $\phi_a(x)=0$ for all $a\in M$. Hence we have
\[
\phi_b=\sum_{a\in M} \phi_a,
\]
which belongs to $V$, as required.
\par
Now choose $T\in \S'$ and $a\colon T\to[q]$ and note that $a\in\F$ if $\abs{T}=t$. We show that $\phi_a$ is in the span of $\{\phi_b:b\in \F'\}$. We proceed with an induction on the number $c$ of elements in $[nr]$ mapped to $q$ under $a$, with the base case being $c=0$. Suppose now that $c$ is nonzero. Then there exists $i_0\in T$ with $a(i_0)=q$. For each $k\in[q-1]$, define $a^k\colon T \to[q]$ by
\[ 
a^k(i)=\begin{cases}
a(i)&\text{for }i\neq i_0\\
k&\text{for } i=i_0
\end{cases}
\]
and let $a'$ be the mapping $a$ restricted to $T\setminus\{i_0\}$. Then we have
\[
\phi_a=\phi_{a'}-\sum_{k=1}^{q-1}\phi_{a^k}.
\]
By the induction hypothesis, the right-hand side is in the span of $\{\phi_b:b\in \F'\}$, which completes the proof.
\end{proof}
\par
For the boundedness of $V^\perp$, it remains to prove that $\{\phi_b:b\in \F'\}$ is locally decodable, from which we can also deduce that this set is linearly independent. For $x\in X$, let $\phi(x)$ be the element of $\ZZ^{\F'}$ with entries $\phi_b(x)$. We will also show that the lattice spanned by the vectors $\{\phi(x):x\in X\}$ equals $\ZZ^ {\F'}$. This property will be helpful later to determine the divisibility constant of $V$.
\begin{lem}\label{lem:locdecbasis}
The set $\{\phi_b:b\in \F'\}$ is a locally decodable basis for $V$ with bound $2^t$. Moreover we have
\begin{align}\label{eq:ZZwdAequalsLattice}
\ZZ^{ \F'}=\left\{\sum_{x\in X} n_x \phi(x): n_x\in\ZZ \right\}.
\end{align}
\end{lem}
\begin{proof}
For $a\colon R\to [q-1]$ and $b\colon T\to [q-1]$ in $\F'$ write $a\preceq b$ if $R\subseteq T$ and $a(i)=b(i)$ for all $i\in R$. This defines a partial order on $\F'$.
\par
We extend each mapping $b\colon T\to [q-1]$ in $\F'$ to a mapping $x^b\colon [nr]\to[q]$ in $X$ via
\[
x^b(i)=\begin{cases}
b(i)&\text{for $i\in T$},\\
q&\text{otherwise}.
\end{cases}
\]
For $a,b\in\F'$, we then have $\phi_{a}(x^b)=1_{a\preceq b}$, where~$1_A$ is the indicator of an event $A$. For each $b\colon T\to[q-1]$ in $\F'$, we define $\gamma_b\colon X\to\ZZ$ by
\[
\gamma_b(x)=\begin{cases}
(-1)^{|T|-|S|}&\text{if $x=x^c$ for some $c\colon S\to[q-1]$ in $\F'$ with $c\preceq b$},\\
0&\text{otherwise}.
\end{cases}
\]
\par
Next we  show that the mappings $\gamma_b$ satisfy~\eqref{eqn:localdecodability} with $m=1$. Note that each $x\in X$ with $\gamma_b(x)\neq 0$ corresponds to exactly one $c\in \F'$ with $c\preceq b$. Hence, for all $a,b\in\F$, we have
\begin{align*}
\sum_{x\in X}\gamma_b(x)\phi_{a}(x)
&=\sum_{c\preceq b} \gamma_b(x^c)\phi_{a}(x^c)\\
&=\sum_{c\preceq b} \gamma_b(x^c) 1_{a\preceq c}\\
&=1_{a\preceq b}\sum_{a\preceq c\preceq b}\gamma_b(x^c).
\end{align*}
Let $a,c,b$ have domains $R,S,T$. Then the summand in the latter sum equals $(-1)^{\abs{T}-\abs{S}}$ and the condition $a\preceq c\preceq b$ means that $R\subseteq S\subseteq T$ and the image of  $S$ under~$c$ is fixed by the image of $S$ under $b$. Hence the mappings $c\in\F'$ satisfying $a\preceq c\preceq b$ are in one-to-one correspondence with sets~$S$ satisfying $R\subseteq S\subseteq T$. There are exactly $\binom{|T|-|R|}{k}$ ways to choose such a subset~$S$ with $\abs{T}-k$ elements and therefore we have
\begin{align}
\sum_{x\in X}\gamma_b(x)\phi_{a}(x)&=1_{a\preceq b}\sum_{k=0}^{|T|-|R|}(-1)^k\binom{|T|-|R|}{k}\notag\\
&=1_{a\preceq b}\cdot 1_{|T|=|R|}\notag\\
&=\delta_{a,b}.   \label{eq:locdec-2}
\end{align}
This establishes~\eqref{eqn:localdecodability} for $m=1$. Let $\phi$ and $\gamma$ be the $X\times \F'$ matrices with $\phi_{x,b}=\phi_b(x)$ and $\gamma_{x,b}=\gamma_b(x)$, respectively. Then~\eqref{eq:locdec-2} implies that $\gamma^T\phi$ is the identity matrix and therefore~$\phi$ has full rank. Together with Lemma~\ref{lem:spanningsetV} it follows that $\{\phi_b:b\in \F'\}$ is a basis for $V$. Since~\eqref{eqn:localdecodability} holds, this basis is locally decodable.
\par
To obtain a bound for the local decodability, note that for each $b\in\F'$, we have
\begin{align*}
\norm{\gamma_b}_1=\sum_{x\in X}\big\lvert\gamma_b(x)\big\rvert=|\{x\in X : \,x=x^c\text{ for some }c\preceq b\}|.
\end{align*}
If $T$ is the domain of $b$, then this number is just the number of subsets of~$T$, namely $2^{\abs{T}}$. Since $\abs{T}\le t$, this shows that $\{\phi_b:b\in\F'\}$ is a locally decodable basis for $V$ with bound $2^t$. 
\par
To prove the second statement of the lemma, note that $\ZZ^{ \F'}$ is equipped with the standard basis $\{e^b:b\in  \F'\}$, where $e_{a}^b=\delta_{a,b}$ for all $a,b\in\F'$. From~\eqref{eq:locdec-2} we have
\[
\sum_{x\in X}\gamma_b(x)\phi(x)=e^b,
\]
which establishes the second statement of the lemma.
\end{proof}
\par
Now note that $V$ has a $c_2$-bounded integer basis in $\ell_\infty$ with $c_2=1$ and $\abs{\F'}\le\abs{\F}$, which follows since $V$ has a basis of size $\F'$ and a spanning set of size $\F$. Hence Lemmas~\ref{lem:locally_decodable_condition} and~\ref{lem:locdecbasis} imply that $V^\perp$ has a $c_3$-bounded integer basis in $\ell_1$ with $c_3=2^{t+1}\abs{\F}$.

\subsection*{(C3) Divisibility} 
For every $b\colon T\to [q-1]$ in $\F'$, we have
\begin{align*}
\frac{1}{\abs{X}}\sum_{x\in X}  \phi_b(x)
&=\frac{1}{\abs{X}}\,|\{x\in X : \text{$x(i)=b(i)$ for all $i\in T$}\}|\\
&=\frac{q^{nr-|T|}}{\abs{X}}=\frac{1}{q^{|T|}}.
\end{align*}
Since $|T|\leq t$, the number $q^t/\abs{X}\sum_{x\in X} \phi_b(x)$ is an integer. From~\eqref{eq:ZZwdAequalsLattice} we conclude that $V$ satisfies the divisibility condition and that the divisibility constant of $V$ is $c_1=q^t$.

\subsection*{Proof of Theorem~\ref{thm:mainresult}}
We have verified the conditions of the KLP theorem with the parameters
\[
c_1=q^t,\quad c_2=1,\quad c_3=2^{t+1}\abs{\F}.
\]
Moreover we have $\dim(V)\le \abs{\F}$ and $\abs{\F}=q^t\abs{\S}$, where
\[
|\S|=|\{(t_1,\dots,t_n):t_i\in\{0,1,\dots,r\},\,\sum_{i=1}^n t_i=t\}|.
\]
This is the number of $r$-restricted partitions of $t$ with at most $n$ parts, which is upper bounded by the number of non-restricted partitions of $t$ with at most~$n$ parts. It is well known and readily verified that this number equals $\binom{n+t-1}{t}$ and hence, by using the standard bound $\binom{m}{k}\le\left(\frac{em}{k}\right)^k$, we obtain
\[
|\S|\leq\binom{n+t-1}{t}\leq\left(\frac{e(n+t)}{t}\right)^t.
\]
The KLP theorem now implies the existence of an ordered orthogonal array~$Y$ of strength $t$ satisfying $|Y|\leq \big(\frac{cq(n+t)}{t}\big)^{ct}$ for some universal constant $c>0$. This proves Theorem~\ref{thm:mainresult}.\hfill\qed



\providecommand{\bysame}{\leavevmode\hbox to3em{\hrulefill}\thinspace}
\providecommand{\MR}{\relax\ifhmode\unskip\space\fi MR }
\providecommand{\MRhref}[2]{%
	\href{http://www.ams.org/mathscinet-getitem?mr=#1}{#2}
}
\providecommand{\href}[2]{#2}

\end{document}